\documentclass[10pt]{amsart}
\usepackage{amsrefs}
\usepackage{amssymb}
\usepackage[all]{xy}

\usepackage{hyperref}

\usepackage{tikz}
\usetikzlibrary{matrix,arrows}

\DeclareMathOperator{\Img}{Im}

\DeclareMathOperator{\Ker}{Ker}

\DeclareMathOperator{\ab}{ab}

\DeclareFontFamily{U}{wncy}{}
\DeclareFontShape{U}{wncy}{m}{n}{<->wncyr10}{}
\DeclareSymbolFont{mcy}{U}{wncy}{m}{n}
\DeclareMathSymbol{\Sha}{\mathord}{mcy}{"58}
\DeclareMathSymbol{\sha}{\mathord}{mcy}{"78}

\begin{document}

\newtheorem{thm}{Theorem}[section]
\newtheorem{cor}[thm]{Corollary}
\newtheorem{lem}[thm]{Lemma}
\newtheorem{fact}[thm]{Fact}
\newtheorem{prop}[thm]{Proposition}
\newtheorem{defin}[thm]{Definition}
\newtheorem{exam}[thm]{Example}
\newtheorem{examples}[thm]{Examples}
\newtheorem{rem}[thm]{Remark}
\newtheorem{case}{\sl Case}
\newtheorem{claim}{Claim}
\newtheorem{question}[thm]{Question}
\newtheorem{conj}[thm]{Conjecture}
\newtheorem*{notation}{Notation}
\swapnumbers
\newtheorem{rems}[thm]{Remarks}
\newtheorem*{acknowledgment}{Acknowledgments}
\newtheorem*{thmno}{Theorem}

\newtheorem{questions}[thm]{Questions}
\numberwithin{equation}{section}

\newcommand{\gr}{\mathrm{gr}}
\newcommand{\inv}{^{-1}}
\newcommand{\isom}{\cong}
\newcommand{\dbC}{\mathbb{C}}
\newcommand{\F}{\mathbb{F}}
\newcommand{\dbN}{\mathbb{N}}
\newcommand{\Q}{\mathbb{Q}}
\newcommand{\dbR}{\mathbb{R}}
\newcommand{\dbU}{\mathbb{U}}
\newcommand{\Z}{\mathbb{Z}}
\newcommand{\calG}{\mathcal{G}}
\newcommand{\calX}{\mathcal{X}}
\newcommand{\K}{\mathbb{K}}
\newcommand{\rmH}{\mathrm{H}}
\newcommand{\bfH}{\mathbf{H}}
\newcommand{\rmr}{\mathrm{r}}
\newcommand{\Span}{\mathrm{Span}}
\newcommand{\eue}{\mathbf{e}}


\newcommand{\hac}{\hat c}
\newcommand{\hatheta}{\hat\theta}

\title[Massey products and Pythagorean fields]{Massey products in Galois cohomology \\ and Pythagorean fields}

\author{Claudio Quadrelli}
\address{Department of Science \& High-Tech, University of Insubria, Como, Italy EU}
\email{claudio.quadrelli@uninsubria.it}
\date{\today}

\begin{abstract}
We prove that a strengthened version of Mina\v{c}--T\^an's Massey Vanishing Conjecture holds true for
fields with a finite number of square classes whose maximal pro-$2$ Galois group is of elementary type (as defined by I.~Efrat).
In particular, this proves Mina\v{c}--T\^an's Massey Vanishing Conjecture for Pythagorean fields with a finite number of square classes and their finite extensions.
\end{abstract}

\subjclass[2010]{Primary 12G05; Secondary 20E18, 20J06, 12F10}

\keywords{Galois cohomology, Massey products, Pythagorean fields, absolute Galois groups, elementary type conjecture}

\maketitle

\section{Introduction}
\label{sec:intro}

One of the topics which gained great interest in Galois theory, in recent years, is the study of {\sl Massey products} in Galois cohomology.
Let $G$ be a pro-$p$ group --- where $p$ is a prime number ---, and let $\F_p$ denote the finite field with $p$ elements considered as a trivial $G$-module.
For $n\geq2$, the {\sl $n$-fold Massey product} with respect to $\F_p$ is a multi-valued map from the first $\F_p$-cohomology group $\rmH^1(G,\F_p)$ to the second $\F_p$-cohomology group $\rmH^2(G,\F_p)$: namely, it associates a sequence of length $n$ of elements $\alpha_1,\ldots,\alpha_n$ (non-necessarily all distinct to each other) of $\rmH^1(G,\F_p)$ to a (possibly empty) subset of $\rmH^2(G,\F_p)$, denoted by $\langle\alpha_1,\ldots,\alpha_n\rangle$.

Given a sequence $\alpha_1,\ldots,\alpha_n$ of elements of $\rmH^1(G,\F_p)$, the associated $n$-fold Massey product $\langle\alpha_1,\ldots,\alpha_n\rangle$ is said to be {\sl defined} if it is not the empty set, and moreover it is said to {\sl vanish} if $0\in\langle\alpha_1,\ldots,\alpha_n\rangle$.
A pro-$p$ group is said to have the {\sl $n$-Massey vanishing property} with respect to $\F_p$ if, for any sequence $\alpha_1,\ldots,\alpha_n$ of elements of $\rmH^1(G,\F_p)$, the $n$-fold Massey product $\langle\alpha_1,\ldots,\alpha_n\rangle$ vanishes whenever it is defined.

Given a field $\K$, let $G_{\K}(p)$ denote the {\sl maximal pro-$p$ Galois group} of $\K$ --- namely, $G_{\K}(p)$ is the maximal pro-$p$ quotient of the absolute Galois group of $\K$.
The recent hectic research on Massey products in Galois cohomology started with the paper \cite{hopwick}, where M.J.~Hopkins and K.G.~Wickelgren proved that if $\K$ is a global field of characteristic not 2, then its maximal pro-2 Galois group $G_{\K}(2)$ has the $3$-Massey vanishing property, and asked whether this is true for any field of characteristic not 2.

Subsequently, J.~Mina\v{c} and N.D.~T\^an conjectured the following (cf. \cite{mt:conj}).

\begin{conj}\label{conj:MvanMT}
 Let $\K$ be a field containing a root of 1 of order $p$.
 Then the maximal pro-$p$ Galois group $G_{\K}(p)$ of $\K$ satisfies the $n$-Massey vanishing property for every $n>2$.
\end{conj}

In the meanwhile, the following remarkable results have been obtained: E.~Matzri proved that $G_{\K}(p)$ has the $3$-Massey vanishing property if $\K$ contains a root of 1 of order $p$ (see the preprint \cite{eli:Massey}, see also the published works \cite{EM:Massey,MT:Masseyall}); J.~Mina\v{c} and N.D.~T\^an proved Conjecture~\ref{conj:MvanMT} for local fields (see \cite{mt:Massey}); Y.~Harpaz and O.~Wittenberg proved Conjecture~\ref{conj:MvanMT} for number fields (see \cite{HW:Massey}); A.~Merkurjev and F.~Scavia proved that the maximal pro-$2$ Galois group of every field satisfies the $4$-Massey vanishing property (see \cite{MerSca1}).
Further interesting results on Massey products in Galois cohomology have been obtained by various authors (see, e.g., \cite{eli2,eli3,wick,jochen,GPM,PJ,LLSWW}).
For an overview on Massey products in Galois cohomology, we direct the reader to \cite{mt:Massey}.

In this work, we focus on a strong version of the Massey vanishing property, introduced by A.~P\'al and E.~Sz\'abo in \cite{pal:Massey}.
If the $n$-fold Massey product $\langle\alpha_1,\ldots,\alpha_n\rangle$, associated to a sequence $\alpha_1,\ldots,\alpha_n$ of elements of the first $\F_p$-cohomology group of a pro-$p$ group $G$, is defined, then necessarily one has
\begin{equation}\label{eq:cup 0}
 \alpha_1\smallsmile\alpha_2=\ldots=\alpha_{n-1}\smallsmile\alpha_n=0,
\end{equation}
where $\textvisiblespace\smallsmile\textvisiblespace$ denotes the {\sl cup-product} (see, e.g., \cite[Rem.~2.2]{cq:massey}).
A pro-$p$ group $G$ is said to have the {\sl strong} $n$-Massey vanishing property if every sequence $\alpha_1,\ldots,\alpha_n$ of elements of $\rmH^1(G,\F_p)$ satisfying \eqref{eq:cup 0} yields an $n$-fold Massey product $\langle\alpha_1,\ldots,\alpha_n\rangle$ which vanishes.
 
It has been show by A.~P\'al and G.~Quick that the maximal pro-2 Galois group of a field with virtual cohomological dimension at most 1 satisfies the strong $n$-Massey vanishing property with respect to $\F_2$ for every $n>2$ (see \cite{palquick:Massey}).
Our main result is to show that the same holds for {\sl Pythagorean fields}.

\begin{thm}\label{thm:main}
 Let $\K$ be a Pythagorean field with finitely many square classes, and let $\mathbb{L}/\K$ be an extension of finite degree.
 Then the maximal pro-2 Galois group $G_{\mathbb{L}}(2)$ of $\mathbb{L}$ satisfies the strong $n$-Massey vanishing property with respect to $\F_2$ for every $n>2$.
\end{thm}

Recall that a field is said to be {\sl Pythagorean} if any sum of two squares is a square.
Equivalently, a field $\K$ is Pythagorean if, and only if, its Witt ring $W(\K)$ is torsion free as an additive abelian group. 
Pythagorean fields play an important role in the study of quadratic forms (see, e.g., \cite{jacob1,jacob2,elmanlam}).

The proof of Theorem~\ref{thm:main} relies on the description of the maximal pro-2 Galois group of a Pythagorean fields with finitely many square classes provided by J.~Mina\v{c} in \cite{minac:pyt}: namely, the maximal pro-2 Galois group of such a field may be constructed starting from cyclic groups of order 2 and iterating a finite number of times free pro-$2$ products and semidirect products with free abelian pro-2 groups with action given by inversion (see also \cite[\S~6.2]{MPQT}).

Such pro-2 groups constitute a subfamily of the family of pro-2 groups of {\sl elementary type}.
For an arbitrary prime $p$, the family of pro-$p$ groups of elementary type (introduced by I.~Efrat, see \cite{efrat:small}) is the family of finitely generated pro-$p$ groups which may be constructed starting from free pro-$p$ groups and {\sl Demushkin} pro-$p$ groups (which include the cyclic group of order 2 if $p=2$) and iterating free pro-$p$ products and certain semidirect products with $\Z_p$ (see Definition~\ref{defin:ET} below for the detailed definition).
In fact, we prove the following more general result.

\begin{thm}\label{thm:ET intro}
Let $G$ be a pro-$2$ group of elementary type.
Then every open subgroup $H$ of $G$ has the strong $n$-Massey vanishing property with respect to $\F_2$ for every $n>2$.
\end{thm}

By the work of J.~Mina\v{c} on the maximal pro-2 Galois group of Pythagorean fields, Theorem~\ref{thm:ET intro} implies Theorem~\ref{thm:main}.
The proof of Theorem~\ref{thm:ET intro} uses a result --- whose original formulation, for discrete groups, is due to W.~Dwyer, see \cite{dwyer} --- which interprets the vanishing of Massey products in the $\F_p$-cohomology of a pro-$p$ group $G$ in terms of the existence of certain upper unitriangular representations of $G$.

Theorem~\ref{thm:ET intro} was proved by the author for every prime $p$ in \cite{cq:massey}, but with restrictions for the case $p=2$ --- for example, excluding cyclic groups of order 2 among the ``building bricks'' of pro-2 groups of elementary type, and thus excluding maximal pro-2 groups of Pythagorean fields.
Therefore, Theorem~\ref{thm:ET intro} --- where such restrictions are dropped --- is complementary to \cite[Thm.~1.2]{cq:massey}.

In the '90s, I.~Efrat formulated the {\sl Elementary Type Conjecture} on maximal pro-$p$ Galois groups,
which predicts that if $\K$ is a field containing a root of 1 of order $p$ and $G_{\K}(p)$ is finitely generated, then $G_{\K}(p)$ is of elementary type, see \cite{ido:etc}.
This is known to be true, for example, if $\K$ is an extension of relative transcendence degree 1 of a pseudo algebraically closed field (see \cite[Ch.~11]{friedjarden} and \cite[\S~5]{ido:Hasse}); or if $\K$ is an algebraic extension of a global field of characteristic not $p$  (see \cite{ido:etc2}); see also Proposition~\ref{prop:etc fields} below.
Hence, Theorem~\ref{thm:ET intro} implies the following.

\begin{cor}\label{cor:etc}
 Let $\K$ be one of the following:
  \begin{itemize}
  \item[(a)] a local field, or an extension of transcendence degree 1 of a local field;
\item[(b)] a pseudo algebraically closed (PAC) field, or an extension of relative transcendence degree 1 of a PAC field;
\item[(c)] a $2$-rigid field {\rm (}for the definition of $p$-rigid fields see \cite[p.~722]{ware}{\rm )};
\item[(d)] an algebraic extension of a global field of characteristic not $p$;
\item[(e)] a valued $2$-Henselian field with residue field $\kappa$, and $G_{\kappa}(2)$ satisfies the strong $n$-Massey vanishing property for every $n>2$.
\end{itemize}
 Suppose further that $\K$ has a finite number of square classes.
 Then $G_{\mathbb{L}}(2)$ has the strong $n$-Massey vanishing property for every $n>2$, for any finite extension $\mathbb{L}/\K$.
 \end{cor}
 
Analogously to Theorem~\ref{thm:ET intro}, Corollary~\ref{cor:etc} has been proved by the author for every prime $p$, but with the further assumption that $\sqrt{-1}\in\K$ if $p=2$ (see \cite[Cor.~1.3]{cq:massey}), which is dropped here.
So, Corollary~\ref{cor:etc} is complementary to \cite[Cor.~1.3]{cq:massey}.

 
{\small \subsection*{Acknowledgments}
The author wishes to thank I.~Efrat, E.~Matzri, J.~Mina\v{c}, F.W.~Pasini, and N.D.~T\^an, for several inspiring discussions on Massey products in Galois cohomology and pro-$p$ groups of elementary type, which occurred in the past years; and the anonymous referee for the helpful comments.
}

 
 \section{Preliminaries}\label{sec:prel}
 Given an arbitrary group $\Gamma$, and two elements $g,h\in \Gamma$, we use the following notation:
 \[
  {}^gh=ghg^{-1}\qquad\text{and}\qquad [g,h]={}^gh\cdot h^{-1}=ghg^{-1}h^{-1}.
 \]
 Henceforth, given a pro-2 group $G$ subgroups will be implicitly assumed to be closed with respect to the pro-2 topology, and generators are to be intended in the topological sense.

\subsection{$\F_p$-cohomology of pro-$p$ groups}
 
  Given a pro-$p$ group $G$, consider $\F_p$ as a trivial $G$-module.
 Here we recall some basic facts on $\F_p$-cohomology of pro-$p$ groups, which will be used throughout the paper.
 As references for $\F_p$-cohomology of pro-$p$ groups, we direct the reader to \cite[Ch.~I, \S~4]{serre:galc} and to \cite[Ch.~I and Ch.~III, \S~3]{nsw:cohn}.
 
 Henceforth, $\rmH^s(G)$ will denote the $s$-th $\F_p$-cohomology group of $G$.
 For the first $\F_p$-cohomology group one has
 \begin{equation}\label{eq:H1}
  \rmH^1(G)=\mathrm{Hom}(G,\F_p)=\left(G/G^p[G,G]\right)^\ast,
 \end{equation}
where the middle term is the $\F_p$-vector space of the homomorphisms of pro-$p$ groups $G\to\F_p$, while in the right-side term $G^p[G,G]$ denotes the Frattini subgroup of $G$, namely, the subgroup of $G$ generated by $\{x^p[y,z]\mid x,y,z\in G\}$, and $\textvisiblespace^\ast$ denotes the $\F_p$-dual space (cf. \cite[Ch.~I, \S~4.2]{serre:galc}). 

In particular, if $G$ is finitely generated, and $\calX=\{x_1,\ldots,x_d\}$ is a minimal generating set of $G$, then 
\[
 \mathcal{B}=\left\{\:\chi_1,\ldots,\chi_d\in \rmH^1(G)\:\mid\:\chi_h(x_k)=\delta_{hk}\text{ for }h,k=1,\ldots,d\:\right\}
\]
is a basis of $\rmH^1(G)$, and it is said to be the basis dual to $\calX$.

The cup-product is a bilinear map
\[
 \xymatrix{\rmH^s(G)\times\rmH^t(G)\ar[r]^-{\smallsmile} &\rmH^{s+t}(G),\qquad s,t\geq0, }
\]
and it is skew-symmetric, i.e., $\beta\smallsmile\alpha=(-1)^{st}\alpha\smallsmile\beta$ for all $\alpha\in\rmH^s(G)$ and $\beta\in\rmH^t(G)$ (cf. \cite[Prop.~1.4.4]{nsw:cohn}).
Thus, in case $p=2$ the cup-product is symmetric, and for $\alpha\in\rmH^1(G,\F_2)$ we write $\alpha^2=\alpha\smallsmile\alpha$.
One has the following (cf. \cite[Lemma~2.4]{idojan:descending}).

\begin{lem}\label{lem:bock}
Let $G$ be a pro-$2$ group. 
Given $\alpha\in\rmH^1(G)$, one has $\alpha^2=0$ if, and only if, there exists a homomorphism of pro-2 groups $\tilde\alpha\colon G\to\Z/4$ such that the diagram
\[
 \xymatrix{ & G\ar[d]^-{\alpha}\ar[dl]_-{\tilde\alpha} \\ \Z/4\ar@{->>}[r]^-{\pi} & \F_2 }
\]
commutes (here $\pi\colon \Z/4\twoheadrightarrow \F_2$ denotes the canonical projection).
\end{lem}

\begin{rem}\label{rem:kummer}\rm
 Let $\K$ be a field containing a root of 1 of order $p$, and let $\K^\times$ denote its multiplicative group.
 By Kummer theory, one has an isomorphism of $\F_p$-vector spaces
 \[\K^\times/(\K^\times)^p\simeq \rmH^1(G_{\K}(p))\]
(cf., e.g., \cite[Ch.~I, \S~1.2, Rem.~2]{serre:galc}).
In particular, if $p=2$ then $G_{\K}(2)$ is finitely generated if, and only if, $\K$ has a finite number of square classes.
 \end{rem}

Now assume that $G$ has a minimal presentation
\begin{equation}\label{eq:minpres}
G=\langle \:x_1,\ldots,x_d\:\mid\:r_1=\ldots=r_m=1\:\rangle
\end{equation}
(for the definition of a presentation, and of a minimal presentation, of a pro-$p$ group see, e.g., \cite[\S~4.6 and p.~311]{ddsms}).
Then the number of defining relations $m$ is equal to the dimension of $\rmH^2(G)$ (cf. \cite[Ch.~I,\S~4.3]{serre:galc}).

Moreover, for $p=2$ one has the following result (which is a special case of \cite[Prop.~3.9.13]{nsw:cohn}), which connects the behavior of the cup-product and the shape of the defining relations.

\begin{prop}\label{prop:rel cup}
Let $G$ be a pro-2 group with minimal presentation \eqref{eq:minpres},
yielding the dual basis $\{\chi_1,\ldots,\chi_d\}$ of $\rmH^1(G)$, and suppose that for each defining relation $r_l$, $l=1,\ldots,m$, one has 
\[
r_l\equiv \prod_{1\leq h<h'}[x_h,x_{h'}]^{a(l)_{h,h'}}\cdot\prod_{1\leq k\leq d}x_{k}^{2b(l)_{k}}\mod G^4\left[G^2,G^2\right],
\]
for some $a(l)_{h,h'},b(l)_k\in\F_2$ --- here $G^4[G^2,G^2]$ denotes the subgroup generated by $\{x^4[y,z]\:\mid\:x\in G,\:y,z\in G^2\}$.
Then for each $l=1,\ldots,m$ there is a linear form $\mathrm{tr}_l\colon \rmH^2(G,\F_2)\to\F_2$ such that
\[
 \mathrm{tr}_l(\chi_h\smallsmile\chi_{h'})=\begin{cases}
                                            a(l)_{h,h'} & \text{if }h<h'\\ b(l)_h & \text{if }h=h',
                                           \end{cases}\]
for all $1\leq h\leq h'\leq d$.
\end{prop}

 
 \subsection{Massey products and upper unitriangular matrices}\label{ssec:massey}
For a survey on Massey products in Galois cohomology we direct the reader to \cite{vogel}, \cite{mt:Massey}, or \cite{cq:massey}. 
Throughout the  paper, we will need only the ``translation'' of Massey products in terms of unipotent representations, recalled here below.

For $n\geq 2$ let
\[
 \dbU_{n+1}=\left\{\left(\begin{array}{ccccc} 1 & a_{1,2} & \cdots & & a_{1,n+1} \\ & 1 & a_{2,3} &  \cdots & \\
 &&\ddots &\ddots& \vdots \\ &&&1& a_{n,n+1} \\ &&&&1 \end{array}\right)\mid a_{i,j}\in\F_p \right\}\subseteq 
 \mathrm{GL}_{n+1}(\F_p)
\]
be the group of unipotent upper-triangular $(n+1)\times(n+1)$-matrices over $\F_p$.
Let $I_{n+1}$ denote the $(n+1)\times(n+1)$ identity matrix, and for $1\leq i<j\leq n+1$, let $E_{ij}$ denote the $(n+1)\times(n+1)$-matrix with 1 at the entry $(i,j)$, and 0 elsewhere.

Now let $G$ be a pro-$p$ group.
For a homomorphism of pro-$p$ groups $\rho\colon G\to\dbU_{n+1}$, and for $1\leq i\leq n$, let $\rho_{i,i+1}$ denote the projection of $\rho$ on the $(i,i+1)$-entry.
Observe that $\rho_{i,i+1}\colon G\to\F_p$ is a homomorphism of pro-$p$ groups, and thus we may consider $\rho_{i,i+1}$ as an element of $\rmH^1(G,\F_p)$.
One has the following ``pro-$p$ translation'' of Dwyer's result on Massey products (cf., e.g., \cite[\S~8]{ido:Massey}).

\begin{prop}\label{prop:masse unip}
Let $G$ be a pro-$p$ group and let $\alpha_1,\ldots,\alpha_n$ be a sequence of elements of $\rmH^1(G,\F_p)$, with $n\geq2$.
The $n$-fold Massey product $\langle\alpha_1,\ldots,\alpha_n\rangle$ vanishes if, and only if, there exists a continuous
homomorphism $ \rho\colon G\to\dbU_{n+1}$ such that $\rho_{i,i+1}=\alpha_i$ for every $i=1,\ldots,n$.
\end{prop}

Employing Proposition~\ref{prop:masse unip}, and the universal properties of free pro-$p$ groups and free pro-$p$ products of pro-$p$ groups respectively, it is fairly easy to prove the following (cf., e.g., \cite[Ex.~4.1]{mt:Massey} and \cite[Th.~5.1]{BCQ:AbsGalType}; for the definition of free pro-$p$ product of pro-$p$ groups see, e.g., \cite[\S~9.1]{ribzal}).

\begin{prop}\label{prop:free}
 \begin{itemize}
  \item[(i)] A free pro-$p$ group has the strong $n$-Massey vanishing property for every $n>2$.
  \item[(ii)] Let $G_1,G_2$ be two pro-$p$ groups with the the strong $n$-Massey vanishing property for every $n>2$.
  Then also the free pro-$p$ product $G_1\ast G_2$ has the strong $n$-Massey vanishing property for every $n>2$.
 \end{itemize}
\end{prop}

\begin{rem}\rm
 If a field $\K$ has characteristic equal to $2$, then the maximal pro-2 Galois group $G_{\K}(2)$ is a free pro-2 group (cf., e.g., \cite[Thm.~6.1.4]{nsw:cohn}), and thus it has the strong $n$-Massey vanishing property for every $n>2$ by Proposition~\ref{prop:free}.
 So, henceforth we may always tacitly assume that every field we deal with has characteristic not 2.
\end{rem}

Proposition~\ref{prop:masse unip} is also useful to show that, in order to check that a pro-$p$ group satisfies the strong $n$-Massey vanishing property, it suffices to verify that every sequence of length at most $n$ of {\sl non-trivial} cohomology elements of degree 1 whose cup-products satisfy the triviality condition \eqref{eq:cup 0} yields a Massey product containing 0, as stated by the following (cf. \cite[Prop.~2.8]{cq:massey}).

 \begin{prop}\label{prop:trivial strong}
Given $N>2$, a pro-$p$ group $G$ satisfies the strong $n$-Massey vanishing property for every $3\leq n\leq N$ if, and only if,  for every $3\leq n\leq N$, every sequence $\alpha_1,\ldots,\alpha_n$ of non-trivial elements of $\rmH^1(G)$ satisfying the triviality condition \eqref{eq:cup 0} yields an $n$-fold Massey product containing $0$. 
 \end{prop}

 Now fix $n\geq2$.
 Here we will write $\dbU$ instead of $\dbU_{n+1}$, and $I$ instead of $I_{n+1}$.
 For every $k\geq1$ let $\dbU_{(k)}$ denote the $k$-th term of the descending central series of $\dbU$: namely, $\dbU_{(k)}$ is the subgroup of $\dbU$ whose elements have $(i,j)$-entries equal to 0 if $0<j-i<k$.
 For example, 
 $$\dbU_{(1)}=\dbU\qquad\text{and}\qquad \dbU_{(n)}=\left\{\:I+E_{1,n+1}a\:\mid\: a\in\F_p\:\right\},$$
 while $\dbU_{(k)}$ is the trivial subgroup if $k>n$.
 
 From \cite[Prop.~2.10]{cq:massey} one deduces the following.
 
 \begin{lem}\label{lemma:BAC}
 Let $A=(a_{i,j})\in\dbU$ a matrix such that $a_{1,2}=a_{2,3}=\ldots=a_{n,n+1}=1$.
 For every $\mu\in\Z_2$ multiple of 4, there exists a matrix $C(\mu)\in\dbU_{(3)}$ such that 
\begin{equation}\label{eq:CAA}
 [C(\mu),A]=A^\mu.
 \end{equation}
 \end{lem}
 
 \begin{proof}
 First, we observe that the order of $A$ is a power of 2 with exponent at least 2, as $\dbU$ is a finite 2-group, and $A^2\neq I$ as $n\geq3$.
 
 Since the $(i,i+1)$- and $(i,i+2)$-entries of $A^4$ are 0, $A^\mu$ lies in the subgroup $\dbU_{(4)}$.
 Then by \cite[Prop.~2.10]{cq:massey} there exists a matrix $C(\mu)\in\dbU_{(4-1)}$ satisfying \eqref{eq:CAA}.
 \end{proof}


 \section{Oriented pro-$p$ groups}\label{sec:or}
 
 \subsection{Orientations of pro-$p$ groups}\label{ssec:or}
 Let 
 \[
  1+p\Z_p=\left\{\:1+p\lambda\:\mid\:\lambda\in\Z_p\:\right\}
 \]
be the multiplicative group of principal units of $\Z_p$.
 Recall that for $p=2$ one has
 $$1+2\Z_2=\{\pm1\}\times(1+4\Z_2)\simeq (\Z/2)\times \Z_2,$$
 where the right-side isomorphism is an isomorphism of abelian pro-2 groups.
A homomorphism of pro-$p$ groups $\theta\colon G\to1+p\Z_p$ is called an {\sl orientation}, and the pair $(G,\theta)$ is called an {\sl oriented pro-$p$ group} (cf. \cite{qw:cyc}).

The most important example of oriented pro-$p$ groups comes from Galois theory.
 
 \begin{exam}\label{exam:cyclo char}\rm
 Let $\K$ be a field containing a root of 1 of order $p$.
 If $\zeta$ is a root of 1 of order a power of $p$ lying in the separable closure of $\K$, then $\zeta$ lies also in the maximal pro-$p$-extension of $\K$, whose Galois group is the maximal pro-$p$ Galois group $G_{\K}(p)$.
 The pro-$p$ group $G_{\K}(p)$ may be endowed with a natural orientation, the {\sl $p$-cyclotomic character} $\theta_{\K}\colon G_{\K}(p)\to1+p\Z_p$, satisfying
\[
 g(\zeta)=\zeta^{\theta_{\K}(g)}\qquad\forall\:g\in G_{\K}(p)
\]
(cf. \cite[\S~4]{eq:kummer}).
The image of $\theta_{\K}$ is $1+p^f\Z_2$, where $f$ is the maximal positive integer such that $\K$ contains the roots of 1 of order $p^f$ --- if such a number does not exists, i.e., if $\K$ contains all roots of 1 of $p$-power order, then $\Img(\theta_{\K})=\{1\}$, and one sets $f=\infty$.

\noindent
Observe that if $p=2$ then $\Img(\theta_{\K})\subseteq1+4\Z_2$ if $\sqrt{-1}\in\K$;
 on the other hand $\Img(\theta_{\K})=\{\pm1\}$ if $\sqrt{-1}\notin \K$ and $\K(\sqrt{-1})$ contains all roots of 1 of order a power of 2.
 \end{exam}

 From now on, given an orientation $\theta\colon G\to1+p\Z_p$ of a pro-$p$ group $G$, the notation $\Img(\theta)=1+p^\infty\Z_p$ will mean that the image of $\theta$ is trivial. 
 
In the family of oriented pro-$p$ groups one has the following two constructions (cf. \cite[\S~3]{efrat:small}).
\begin{itemize}
 \item[(a)] Let $(G_0,\theta)$ be an oriented pro-2 group, and let $Z$ be a free abelian pro-2 group.
 The {\sl semidirect product} $(Z\rtimes_\theta G_0,\tilde\theta)$ is the oriented pro-$p$ group where $Z\rtimes_\theta G_0$ is the semidirect product of pro-$p$ groups with action $gzg^{-1}=z^{\theta(g)}$ for every $g\in G_0$ and $z\in Z$, and where 
 $$\tilde\theta\colon Z\rtimes_\theta G_0\longrightarrow 1+p\Z_p$$
 is the orientation induced by $\theta$, i.e., $\tilde\theta=\theta\circ\pi$, where $\pi\colon Z\rtimes_\theta G_0\to G_0$ is the canonical projection.
 \item[(b)] Let $(G_1,\theta_1),(G_2,\theta_2)$ be two oriented pro-$p$ groups.
 The {\sl free product} $(G_1\ast G_2,\theta)$ is the oriented pro-$p$ group, while $$\theta\colon G_1\ast G_2\longrightarrow1+p\Z_p$$
 is the orientation induced by the orientations $\theta_1,\theta_2$ via the universal property of the free pro-$p$ product.
\end{itemize}

\subsection{Demushkin pro-2 groups and orientations}\label{ssec:Demushkin or}
Another important source of oriented pro-$p$ groups are Demushkin pro-$p$ groups.
A Demushkin group $G$ is a pro-$p$ group whose $\F_p$-cohomology algebra satisfies the following three conditions:
\begin{itemize}
 \item[(i)] the dimension of $\rmH^1(G)$ is finite, i.e., $G$ is finitely generated;
 \item[(ii)] the dimension of $\rmH^2(G)$ is one, i.e., $G$ has one defining relation;
 \item[(iii)] the cup-product induces a non-degenerate pairing
 \[
 \xymatrix{ \rmH^1(G)\times \rmH^1(G)\ar[r]^-{\smallsmile} & \rmH^2(G) }
 \]
\end{itemize}
(cf., e.g., \cite[Def.~3.9.9]{nsw:cohn}).
Demushkin groups have been studied and classified by S.P.~Demu\v{s}kin, J-P.~Serre and J.P.~Labute (cf. \cite{Demushkin, serre:Demushkin, labute}). 
In particular, Serre proved that every Demushkin group may be endowed with a canonical orientation $\theta_G\colon G\to1+p\Z_p$.

By condition~(ii) above, a Demushkin pro-$p$ group $G$ has a single defining relation, namely,
\begin{equation}\label{eq:pres Demushkin}
 G=\left\langle\:x_1,\ldots,x_d\:\mid\:r=1\:\right\rangle,\qquad d=\dim(\rmH^1(G)),
\end{equation}
yielding the dual bases $\{\chi_1,\ldots,\chi_d\}$ of $\rmH^1(G)$.

Henceforth we focus on Demushkin pro-2 groups.
Altogether, one has four cases, depending on whether $d$ is even or odd, and on the shape of the defining relation $r$
(cf. \cite{labute}).


\subsubsection{First case: $\Img(\theta)\subseteq1+4\Z_2$}
The number $d$ is even and the defining relation is
 \[
 r=[x_1,x_2]x_2^{-2^f}[x_3,x_4]\cdots[x_{d-1},x_d],
 \]
for some $f\in\{2,3,4,\ldots,\infty\}$.
The canonical orientation $\theta_G\colon G\to1+2\Z_2$ is defined by
$$\theta_G(x_1)=1+2^f\qquad\text{and}\qquad \theta_G(x_h)=1\text{ for }h=2,\ldots,d;$$
thus $\Img(\theta_G)=1+2^f\Z_2$.
This is the only case with $\Img(\theta_G)\subseteq 1+4\Z_2$, and the only case of Demushkin pro-2 groups taken into consideration in \cite{cq:massey}.

\noindent
Moreover, $\rmH^2(G)$ is generated by $\chi_1^2$, and one has the relations
\[
 \chi_1^2=\chi_2\smallsmile\chi_3=\ldots=\chi_{d-1}\smallsmile\chi_{d}
\]
and $\chi_{h}\smallsmile\chi_{h'}=0$ for any other pair $(h,h')$, with $h\leq h'$, different to $(1,2)$, $(3,4)$, \ldots, $(d-1,d)$.

\begin{rem}\rm
 Usually the defining relations of a Demushkin pro-$p$ group are described adopting the convention $[x,y]=x^{-1}y^{-1}xy$ (this is the case, for example, of \cite{labute} and \cite[Ch.~III, \S~9]{nsw:cohn}).
 According to the aforementioned convention, the defining relation of this case reads
 \[
  r=x_2^{2^f}[x_2,x_1][x_4,x_3]\cdots[x_{d},x_{d-1}],  \]
and analogously for the next cases.
\end{rem}

\subsubsection{Second case: $-1\in\Img(\theta)$, $d$ even}
The number $d$ is even and the defining relation is
 \[
  r=[x_1,x_2]x_2^{2}[x_3,x_4]x_4^{-2^f}[x_5,x_6]\cdots[x_{d-1},x_d],
 \]
for some $f\in\{2,3,4,\ldots,\infty\}$.
The canonical orientation $\theta_G\colon G\to1+2\Z_2$ is defined by
$$\theta_G(x_1)=-1,\qquad \theta_G(x_3)=1+2^f,\qquad\text{and}\qquad\theta_G(x_h)=1\text{ for }h=2,4,\ldots,d;$$
thus $\Img(\theta_G)=\{\pm1\}\times(1+2^f\Z_2)$.

\noindent
Moreover, $\rmH^2(G)$ is generated by $\chi_1\smallsmile\chi_2$, and one has the relations
\[
 \chi_1\smallsmile\chi_2=\chi_2^2=\chi_3\smallsmile\chi_4=\chi_5\smallsmile\chi_6=\ldots=\chi_{d-1}\smallsmile\chi_{d}
\]
and $\chi_{h}\smallsmile\chi_{h'}=0$ for any other pair $(h,h')$, with $h\leq h'$, different to $(1,2)$, $(2,2)$, $(3,4)$, \ldots, $(d-1,d)$.

\subsubsection{Third case: $\Img(\theta)=\langle-1+2^f\rangle$}

The number $d$ is even and 
the defining relation is 
\[r=[x_1,x_2]x_2^{2-2^f}[x_3,x_4]\cdots[x_{d-1},x_d],
 \]
for some $f\in\{2,3,4,\ldots,\infty\}$.
The canonical orientation $\theta_G\colon G\to1+2\Z_2$ is defined by
$$\theta_G(x_1)=-1+2^f\qquad\text{and}\qquad\theta_G(x_h)=1\text{ for }h=2,\ldots,d;$$
thus $\Img(\theta_G)=\langle-1+2^f\rangle\simeq\Z_2$.

\noindent
Moreover, $\rmH^2(G)$ is generated by $\chi_1\smallsmile\chi_2$, and one has the relations
\[
 \chi_1\smallsmile\chi_2=\chi_2^2=\chi_3\smallsmile\chi_4=\ldots=\chi_{d-1}\smallsmile\chi_{d}
\]
and $\chi_{h}\smallsmile\chi_{h'}=0$ for any other pair $(h,h')$, with $h\leq h'$, different to $(1,2)$, $(2,2)$, $(3,4)$, \ldots, $(d-1,d)$.

\subsubsection{Fourth case: $d$ odd}\label{sssec:C2}

The number $d$ is odd and the defining relation is
\[
  r=x_1^2[x_2,x_3]x_3^{-2^f}[x_4,x_5]\cdots[x_{d-1},x_d],
 \]
for some $f\in\{2,3,4,\ldots,\infty\}$.
Observe that if $d=1$ then $G$ is a cyclic group of order 2 (in this case we set implicitly $f=\infty$).
The canonical orientation $\theta_G\colon G\to1+2\Z_2$ is defined by
$$\theta_G(x_1)=-1,\qquad \theta_G(x_2)=1+2^f,\qquad\text{and}\qquad\theta_G(x_h)=1\text{ for }h=3,\ldots,d;$$
thus $\Img(\theta_G)=\{\pm1\}\times(1+2^f\Z_2)$.

\noindent
Moreover, $\rmH^2(G)$ is generated by $\chi_1\smallsmile\chi_2$, and one has the relations
\[
\chi_1^2= \chi_2\smallsmile\chi_3= \chi_4\smallsmile\chi_5=\ldots=\chi_{d-1}\smallsmile\chi_{d}
\]
and $\chi_{h}\smallsmile\chi_{h'}=0$ for any other pair $(h,h')$, with $h\leq h'$, different to $(1,1)$, $(2,3)$, $(4,5)$, \ldots, $(d-1,d)$.

 \begin{rem}\label{rem:Demushkin}\rm
If $\K$ is an $\ell$-adic local field, with $\ell$ an odd prime, 
then its maximal pro-$2$ Galois group $G_{\K}(2)$ is a 2-generated Demushkin pro-2 group (cf. \cite[Prop.~7.5.9]{nsw:cohn}).
On the other hand, if $\K$ is a 2-adic local field, then its maximal pro-$2$ Galois group $G_{\K}(2)$ is a Demushkin pro-2 group  
with $d= [\K:\Q_p]+2$ generators (cf. \cite[Thm.~7.5.11]{nsw:cohn}).
Finally, $\Z/2$ (which is the Demushkin pro-2 group with defining relation (d) and $d=1$) is the maximal pro-$2$ Galois group of $\K=\dbR$.
It is not known whether {\sl every} Demushkin pro-2 group occurs as the maximal pro-$2$ Galois group of a field --- for example, this is not known for \[
 G=\langle\:x_1,x_2,x_3\:\mid\:x_1^2[x_2,x_3]=1\:\rangle,
\]
 (cf. \cite[Rem.~5.5]{jacobware}, see also \cite[Rem~3.3]{cq:massey}).
 Yet, if a Demushkin pro-2 group does, then necessarily the canonical orientation coincides with the 2-cyclotomic character (cf. \cite[Thm.~4]{labute}). 
\end{rem}

\subsection{Oriented pro-$p$ groups of elementary type}\label{ssec:etc}

The following definition is due to I.~Efrat.

\begin{defin}\label{defin:ET}\rm
  The family $\mathcal{E}_p$ of {\sl oriented pro-$p$ groups of elementary type} is the smallest family of oriented pro-$p$ groups containing
 \begin{itemize}
  \item[(a)] every oriented pro-$p$ group $(G,\theta)$, where $G$ is a finitely generated free pro-$p$ group, and $\theta\colon F\to1+p\Z_p$ is arbitrary (including $G=\{1\}$),
  \item[(b)] every Demushkin pro-$p$ group together with its canonical orientation $(G,\theta_G)$;
 \end{itemize}
and such that 
\begin{itemize}
 \item[(c)] if $(G_0,\theta)$ is an oriented pro-$p$ group of elementary type, then also the semidirect product $(\Z_p\rtimes_\theta G_0,\tilde\theta)$, as defined in \S~\ref{ssec:or}, is in $\mathcal{E}_p$,
 \item[(d)]if $(G_1,\theta_1),(G_2,\theta_2)$ are two oriented pro-$p$ groups of elementary type, then also the free product $(G_1\ast G_2,\theta)$, as defined in \S~\ref{ssec:or}, is in $\mathcal{E}_p$.
\end{itemize}
\end{defin}

\begin{rem}\label{rem:subgroups et}\rm
If $(G,\theta)$ is an oriented pro-2 group of elementary type, and $H$ is a finitely generated subgroup of $G$, then also the oriented pro-2 group $(H,\theta\vert_H)$ is of elementary type (cf., e.g., \cite[Rem.~5.10--(b)]{qw:cyc}).
\end{rem}

\begin{exam}\label{exam:et}\rm
The oriented pro-2 group
\[
 (G,\theta)=Z_2\rtimes_{\theta_2'}\left((Z_1\rtimes_{\theta_1'}((G_1,\theta_1)\ast(G_2,\theta_2)))\ast(G_3,\theta_3)\right)
\]
--- where $G_1$ and $G_3$ are Demushkin pro-2 groups,
 \[
  G_1=\left\langle\:x_1,x_2,x_3\:\mid\:x_1^2[x_2,x_3]=1\:\right\rangle,\qquad
  G_3=\left\langle\:x_5,x_6\:\mid\:[x_5,x_6]x_6^{-2}=1\:\right\rangle,
 \]
$G_2=\langle x_4\rangle\simeq\Z_2$ is free on one generator with $\theta_2(x_4)=1+4$, and $Z_i=\langle z_i\rangle\simeq\Z_2$ ---
is an oriented pro-2 group of elementary type.
\end{exam}

The Elementary Type Conjecture asks whether $(G_{\K}(p),\theta_{\K})\in\mathcal{E}_p$ for every field $\K$ containing a root of 1 of order $p$ such that $[\K^\times:(\K^\times)^p]<\infty$ (cf. \cite{ido:etc}, see also \cite[Question~4.8]{ido:etc2}, \cite[\S~10]{marshall}, \cite[\S~7.5]{qw:cyc} and \cite[Conj.~4.8]{MPQT}).

Now let $\mathcal{EE}_2$ be the smallest subfamily of $\mathcal{E}_2$ containing the trivial group $\{1\}$ with the trivial orientation, the oriented pro-2 group $(G,\theta_G)$ with $G$ the cyclic group of order 2 and $\theta_G\colon G\to\{\pm1\}$ --- namely, $G$ is the Demushkin pro-2 group with $d=1$ ---; and closed with respect to semidirect products with $\Z_2$, and free pro-2 groups.
In other words, $\mathcal{EE}_2$ consists of those oriented pro-2 groups of elementary type obtained using only cyclic groups of order 2 as ``building bricks''.

By the work of J.~Mina\v{c} \cite{minac:pyt}, one has the following (see also \cite[\S~6.2]{MPQT}).

\begin{prop}\label{prop}
 Let $\K$ be a field with a finite number of square classes.
 Then $\K$ is Pythagorean if, and only if, $(G_{\K}(2),\theta_{\K})\in\mathcal{EE}_2$.
\end{prop}

\begin{rem}\label{rem:pyt}\rm
A Pythagorean field $\K$ such that $\sqrt{-1}\notin\K$ (and thus such that $\Img(\theta_{\K})\not\subseteq1+4\Z_2$, cf. Ex.~\ref{exam:cyclo char}) is {\sl formally real}.
On the other hand, if $\sqrt{-1}\in\K$, then $(\K^\times)^2=\K^\times$ (cf. \cite[p.~255]{lam}), and thus $G_{\K}(2)$ is trivial.
\end{rem}

Altogether, for $p=2$ one has the following (cf., e.g., \cite[\S~6]{MPQT}).

\begin{prop}\label{prop:etc fields}
 Let $\K$ be a field with a finite number of square classes.
 Then $(G_{\K}(2),\theta_{\K})\in\mathcal{E}_2$ in the following cases:
\begin{itemize}
  \item[(i)] $\K$ is finite;
\item[(ii)] $\K$ is a PAC field, or an extension of relative transcendence degree 1 of a PAC field;
\item[(iii)] $\K$ is a local field, or an extension of transcendence degree 1 of a local field, with characteristic not $p$;
\item[(iv)] $\K$ is $2$-rigid {\rm (}cf. \cite[p.~722]{ware}{\rm)};
\item[(v)] $\K$ is an algebraic extension of a global field of characteristic not $2$;
\item[(vi)] $\K$ is a valued $2$-Henselian field with residue field $\kappa$, and $(G_{\kappa}(p),\theta_{\kappa})\in\mathcal{E}_2$;
\item[(vii)] $\K$ is a Pythagorean field.
 \end{itemize}
\end{prop}


 \section{Massey products and pro-2 groups of elementary type}\label{sec:main}

\subsection{Demushkin pro-2 groups and Massey products}\label{ssec:Demushkin}

A.~P\'al and E.~Szab\'o proved that for every prime $p$, every Demushkin pro-$p$ group satisfies the strong $n$-Massey vanishing property with respect to $p$ for every $n>2$.(cf. \cite[Thm.~3.5]{pal:Massey}).
Thus, for $p=2$ one has the following.

\begin{prop}\label{prop:Demushkin}
 Every Demushkin pro-2 group has the strong $n$-Massey vanishing property with respect to 2 for every $n>2$.
\end{prop}

In particular, the cyclic group $G=\langle x\mid x^2=1\rangle$ of order 2 --- which is a Demushkin pro-2 group, cf. \S~\ref{sssec:C2} --- has the strong $n$-Massey vanishing property with respect to 2 for every $n>2$.
Actually, this is easily shown using Proposition~\ref{prop:masse unip}.
Indeed, let $\chi\colon G\to\F_2$ the non-trivial homomorphism --- i.e., $\rmH^1(G)=\F_2.\chi$.
Then $\chi^2$ is not trivial by Lemma~\ref{lem:bock}.
Hence, a sequence $\alpha_1,\ldots,\alpha_n$ of elements of $\rmH^1(G)$ satisfies \eqref{eq:cup 0} if, and only if, given two consecutive terms $\alpha_i,\alpha_{i+1}$, they are not both equal to $\chi$, namely, at least one of them is the 0-map.
Hence, the matrix 
\[
 A=I_{n+1}+\alpha_1(x)E_{1,2}+\alpha_2(x)E_{2,3}+\ldots+\alpha_n(x)E_{n,n+1}\in\dbU_{n+1}
\]
has order 2, and the assignment $x\mapsto A$ yields a homomorphism $\rho\colon G\to\dbU_{n+1}$ satisfying the properties prescribed in Proposition~\ref{prop:masse unip}.

\subsection{The infinite dihedral pro-2 group}\label{ssec:dihedral}

Let 
$$G_1=\left\langle\: x_1\:\mid\: x_1^2=1\:\right\rangle\qquad\text{and}\qquad 
G_2=\left\langle\: x_2\:\mid\: x_2^2=1\:\right\rangle$$
be two cyclic groups of order 2.
Their free pro-2 product $G_1\ast G_2$ is the {\sl infinite dihedral pro-2 group} (i.e., the pro-2 completion of the abstract infinite dihedral group)
\[
 D_{2^\infty}=\left\langle\: x_1,x_2\:\mid\: x_1^2=x_2^2=1\:\right\rangle=\left\langle\: x,y\:\mid x^2=[x,y]y^2=1\:\right\rangle,
\]
with $x=x_1$ and $y=x_1x_2$.
In particular, if for $i=1,2$ we endow $G_i$ with the non-trivial orientation $\theta_i\colon G_i\to\{\pm1\}$, then
$$(D_{2^\infty},\theta_{D_{2^\infty}}):=(G_1,\theta_1)\ast(G_2,\theta_2)\simeq Z\rtimes_{\theta_1}(G_1,\theta_1), $$
with $Z=\langle x_1x_2\rangle\simeq\Z_2$.
Thus, the oriented pro-$2$ group $(D_{2^\infty},\theta_{D_{2^\infty}})$ is of elementary type.

Let $\{\chi_1,\chi_2\}$ and $\{\chi,\psi\}$ be the bases of $\rmH^1(D_{2^\infty})$ dual respectively to $\{x_1,x_2\}$ and to $\{x,y\}$.
Then $\chi=\chi_1+\chi_2$ and $\psi=\chi_2$.
Moreover, by \cite[Thm.~4.1.4]{nsw:cohn} one has 
\[
 \rmH^2(D_{2^\infty})=\rmH^2(G_1)\oplus\rmH^2(G_2)=\F_2.\chi_1^2\oplus\F_2.\chi_2^2,
\]
and hence $\{\chi^2,\chi\smallsmile\psi=\psi^2\}$ is another basis of $\rmH^2(D_{2^\infty})$.

Since both $G_1,G_2$ have the strong $n$-Massey vanishing property for every $n>2$, by Proposition~\ref{prop:free} also $D_{2^\infty}$ has the strong $n$-Massey vanishing property for every $n>2$.
This implies the following.

\begin{lem}\label{lemma:B2 I}
 For $n>2$, there exist matrices $A_1,A_2,B_1,B_2\in\dbU_{n+1}$ such that the $(i,i+1)$-entries of both $A_1$ and $A_2$ are equal to 1, for all $i=1,\ldots,n$, and 
 \[
  B_1=\left(\begin{array}{cccccc}1 & 1 & &&&\ast \\ &1&0&&& \\ &&1&1&& \\ &&&1&0& \\ &&&&\ddots&\ddots \\ &&&&&1
            \end{array}\right),\qquad
  B_2=\left(\begin{array}{cccccc}1 & 0 & &&&\ast \\ &1&1&&& \\ &&1&0&& \\ &&&1&1& \\ &&&&\ddots&\ddots \\ &&&&&1
            \end{array}\right) \]
satisfying $[B_1,A_1]=A_1^{-2}$ and $[B_2,A_2]=A_2^{-2}$.
\end{lem}

\begin{proof}
 Consider the sequences $\alpha_1,\ldots,\alpha_n$ and $\alpha'_1,\ldots,\alpha'_n$ of elements of $\rmH^1(D_{2^\infty})$ of length $n$,
 with $\alpha_i=\alpha_j'=\chi+\psi$, for $i$ odd and $j$ even, and $\alpha_i=\alpha_j'=\psi$ for $i$ even and $j$ odd.
 Then both sequences satisfy \eqref{eq:cup 0}, as 
 $$(\chi+\psi)\smallsmile\psi=\chi\smallsmile\psi+\psi^2=0.$$
Since $D_{2^\infty}$ has the strong $n$-Massey vanishing property, by Proposition~\ref{prop:masse unip} there are two continuous homomorphisms $\rho,\rho'\colon D_{2^\infty}\to\dbU_{n+1}$ satisfying $\rho_{i,i+1}=\alpha_i$ and $\rho'_{i,i+1}=\alpha'_i$ for all $i=1,\ldots,n$.
Set $A_1=\rho(y)$ and $A_2=\rho'(y)$, and $B_1=\rho(x)$ and $B_2=\rho'(x)$.
Then these matrices satisfy all the desired properties, as $[x,y]=y^{-2}$, and $\alpha_i(y)=\alpha'_i(y)=\psi(y)=1$ for all $i$, while $\alpha_1(x)=\alpha_2'(x)=\chi(x)=1$, and so on. 
 \end{proof}

\subsection{Semidirect products}

The goal of this subsection is to show that if $(G_0,\theta_0)$ is an oriented pro-2 group of elementary type with $\Img(\theta_0)\not\subseteq1+4\Z_2$ such that $G_0$ has the strong $n$-Massey vanishing property for every $n>2$, then the same holds for $Z\rtimes_{\theta_0}G_0$, with $Z\simeq\Z_2$.

Set $(G,\theta)=Z\rtimes_{\theta_0}(G_0,\theta_0)$, and let $z\in Z$ be a generator, and $\psi\colon Z\to\F_2$ the homomorphism dual to $\{z\}$ --- namely, $\psi(z)=1$, so that $\psi$ generates $\rmH^1(Z)$.
Then it is well-known that
\begin{equation}\label{eq:H1 semidirect}
 \rmH^1(G)=\rmH^1(G_0)\oplus\rmH^1(Z),
\end{equation}
while by the work of A.~Wadsworth (cf.  \cite[Cor.~3.4 and Thm.~3.6]{wadsworth}), one knows that
\begin{equation}\label{eq:H2 semidirect}
 \rmH^2(G)=\rmH^2(G_0)\oplus\left(\rmH^1(G_0)\smallsmile\psi\right),
\end{equation}
where the right-side summand denotes the subspace $\{\alpha\smallsmile\psi\:\mid\:\alpha\in\rmH^1(G_0)\}$,
and one has $\alpha\smallsmile\psi=0$ if, and only if, $\alpha=0$.
In particular, if $\{\chi_1,\ldots,\chi_d\}$ is a basis of $\rmH^1(G_0)$, then
$\{\chi_1\smallsmile\psi,\ldots,\chi_d\smallsmile\psi\}$ is a basis of $\rmH^1(G_0)\smallsmile\psi$.

\begin{rem}\label{rem:semid}\rm
By considering the elements of $\rmH^1(G_0)$ and of $\rmH^1(Z)$ as elements of $\rmH^1(G)$ too, we commit a slight abuse of notation: actually, if $\alpha\in\rmH^1(G_0)$, then we identify it with the element $\tilde\alpha$ of $\rmH^1(G)$ such that $\tilde\alpha\vert_{G_0}=\alpha$, and $\tilde\alpha\vert_Z$ is the 0-map, and analogously if $\alpha\in\rmH^1(Z)$.
\end{rem}

Since $(G_0,\theta_0)$ is a pro-2 group of elementary type, we take a minimal generating set $\calX$ of $G_0$ consisting of the union of the minimal generating sets of the building bricks --- free pro-2 groups and Demushkin pro-2 groups --- of $G_0$, together with generators of the pro-2-cyclic factors of the occurring semidirect products.
We write $\calX$ as follows:
\[
 \calX=\left\{\:u_1,\:\ldots,\:u_s,\:v_1,\:\ldots,\:v_t\:\right\},
\]
where
\begin{itemize}
 \item[(a)] $\theta_0(u_h)\equiv -1\bmod 4$ for all $h=1,\ldots,s$,
 \item[(b)] $\theta_0(v_l)\equiv 1\bmod 4$ for all $l=1,\ldots,t$. 
\end{itemize}
In particular, every generator $u_h$ is the first generator --- called $x_1$ in \S~\ref{ssec:Demushkin or} --- of a ``Demushkin brick'' of $G_0$ with defining relation different to case~(a), or a generator of a ``free brick'' (whose image via the orientation of the brick is equivalent to $-1$ mod 4).

The minimal generating set $\calX$ yields the dual basis 
$\mathcal{B}=\{\omega_1,\ldots,\omega_s,\varphi_1,\ldots,\varphi_t\}$ of $\rmH^1(G_0)$.

As a set of defining relations of $G_0$ we take the union of the defining relations of the ``Demushkin bricks'', together with the relations defining the action in the occurring semidirect products.

Altogether, a generator $u_h$ may show up in a defining relation of the following types:
\begin{equation}\label{eq:rel semidirect u}
 \begin{split}
 & u_h^2[v_{l_1},v_{l_2}]v_{l_2}^{-2^f}[v_{l_3},v_{l_4}]\cdots=1,\\
 & [u_h,v_{l_1}]v_{l_1}^2[v_{l_2},v_{l_3}]v_{l_3}^{-2^f}\cdots=1,\\
 & [u_h,v_{l_1}]v_{l_1}^{2-2^f}[v_{l_2},v_{l_3}]\cdots=1;
 \end{split}
\end{equation}
(where the first relation might be just $u_h^2=1$) and in at most one of the above not involving precisely two generators.
Analogously, a generator $v_l$ with $\theta_0(v_l)=1+2^f$, $2\leq f<\infty$, may show up in a defining relation of the type
\begin{equation}\label{rel semidirec v1}
 [v_l,v_{l_1}]v_{l_1}^{-2^f}[v_{l_2},v_{l_3}]\cdots =1,
\end{equation}
and in at most one of this involving more than two generators.
On the other hand, the defining relations involving the generators $v_l$'s with $\theta_0(v_l)=1$ are just products of elementary commutators involving these generators, namely
\begin{equation}\label{eq:rel semidirect v2}
 [v_{l_1},v_{l_2}]\cdots[v_{l_{m-1}},v_{l_m}]=1.
\end{equation}

\begin{rem}\label{rem:factors relations semidirect}\rm
 Every defining relation of $G_0$ is the product of factors 
 \begin{equation}\label{eq:factor rel}
  [y,v_l]v_l^{1-\theta_0(y)}\qquad \text{with }y\in\calX,\:v_l\in\Ker(\theta_0),
 \end{equation}
and, possibly, $u_h^2$ with $\theta_0(u_h)=-1$.
\end{rem}

\begin{exam}\rm
 Let $(G_0,\theta_0)$ be the oriented pro-2 group of elementary type constructed in Example~\ref{exam:et}.
 We may pick the following minimal generating set of $G_0$:
 \[
  \calX=\left\{\:u_1=x_1,u_2=x_5,v_1=x_2,v_2=x_3,v_3=x_4,v_4=x_6,v_5=z_1,v_6=z_2\:\right\}.
 \]
The generator $u_1$ is involved in the defining relations
\[
 u_1^2[v_1,v_2]=[u_1,v_5]v_5^{2}=[u_1,v_6]v_6^{2}=1,
\]
the generator $u_2$ is involved in the defining relations
\[
 [u_2,v_4]v_4^{2}=[u_2,v_6]v_6^2=1,
\]
while the generator $v_3$ is involved in the defining relations
\[
 [v_3,v_5]v_5^{-4}=[v_3,v_6]v_6^{-4}=1.
\]
\end{exam}

Now consider $G=Z\rtimes_{\theta_0}G_0$.
Proposition~\ref{prop:rel cup} implies that 
\[
 \psi^2=(\omega_1+\ldots+\omega_s)\smallsmile\psi\neq0,
\]
as $[u_1,z]z^{2}=\ldots=[u_s,z]z^2=1$.
In order to proceed toward our result, we change the minimal generating set of $G_0$.
Put $w_1=u_1$, and $w_h=u_1u_h$ for all $h=2,\ldots,s$.
Then 
\[
 \calX'=\left\{\:w_1,\:\ldots,\:w_s,\:v_1,\:\ldots,\:v_t\:\right\}
\]
is again a minimal generating set of $G_0$.
Moreover, put $\chi_0=\omega_1+\ldots+\omega_s$.
Then $\chi_0(w_1)=1$ and 
$$\chi_0(w_h)=\omega_1(w_h)+\omega_h(w_h)=\omega_1(u_1)+\omega_h(u_h)=1+1=0$$
for all $h=2,\ldots,s$, while for $h=2,\ldots,s$ one has $\omega_h(w_1)=\omega_h(u_1)=0$, and
\[
 \omega_h(w_{h'})=\begin{cases}
                   \omega_h(u_1)+\omega_h(u_h)=0+1=1 & \text{if }h'=h,\\
                   \omega_h(u_1)+\omega_h(u_{h'})=0+0=0 & \text{if }h'\neq h.                  \end{cases}\]
Therefore, $\{ \chi_0, \omega_2, \ldots, \omega_s, \varphi_1, \ldots, \varphi_t \}$ is the basis of $\rmH^1(G_0)$ dual to $\calX'$.

\begin{lem}\label{lem:semidirect alpha}
 Let $\alpha,\alpha'$ be elements of $\rmH^1(G)$, $\alpha,\alpha'\neq0$ (and possibly $\alpha=\alpha'$), such that $\alpha\smallsmile\alpha'=0$.
 Then either $\alpha,\alpha'\in\rmH^1(G_0)$, or $\alpha'=\alpha+\chi_0$ and $\alpha(z)=\alpha'(z)=1$.
\end{lem}

\begin{proof}
 We keep the same notation as above.
We may write
\[
 \alpha=a\chi_0+\beta+b\psi\qquad\text{and}\qquad\alpha=a'\chi_0+\beta'+b'\psi,
\]
with $a,b,a',b'\in\F_2$ and $\beta,\beta'\in\Span(\omega_2,\ldots,\omega_s,\varphi_1,\ldots,\varphi_t)$, in a unique way.
Then we compute
\[
 \begin{split}
  0 &=\alpha\smallsmile\alpha'\\
  &=\left(aa'\chi_0^2+(a\beta'+a'\beta)\smallsmile\chi_0+\beta\smallsmile\beta'\right)+ \\ 
&\qquad\qquad\qquad\left((b\beta'+b'\beta)\smallsmile\psi+(ab'+a'b)\chi_0\smallsmile\psi+bb'\psi^2\right)\\
  &=\underbrace{\left((aa'\chi_0+a\beta'+a'\beta)\smallsmile\chi_0+\beta\smallsmile\beta'\right)}_{\in\rmH^2(G_0)}+
\underbrace{\left((b\beta'+b'\beta)+(ab'+a'b+bb')\chi_0\right)\smallsmile\psi}_{\in\rmH^1(G_0)\smallsmile\psi},
  \end{split}\]
as $\psi^2=\chi_0\smallsmile\psi$.  
Therefore both summands are 0.
In particular, the right-side summand is 0 if, and only if, 
\begin{equation}\label{eq:system}
 \begin{cases}
b\beta'+b'\beta=0   \\ab'+a'b+bb'=0,
  \end{cases}
\end{equation}
as  $\beta,\beta'\in\Span(\omega_2,\ldots,\varphi_t)$.
If $b,b'=0$ then $\alpha,\alpha'\in\rmH^1(G_0)$, and the system \eqref{eq:system} is satisfied.
So assume that not both $b,b'$ are 0.

Suppose first that $b=1$ and $b'=0$.
Then the fist equation of \eqref{eq:system} implies $\beta'=0$, and the second equation implies $a'=0$, namely, $\alpha'=0$, a contradiction.
Analogously, the case $b=0$ and $b'=1$ leads to $\alpha=0$, a contradiction.
Therefore, $b,b'=1$, and the fist equation of \eqref{eq:system} implies $\beta=\beta'$, while the second equation implies $a'=a+1$.
Hence $$\alpha'=(a+1)\chi_0+\beta+\psi=\alpha+\chi_0,$$
and $\alpha(z)=\alpha'(z)=\psi(z)=1$.
\end{proof}

We are ready to prove the following.

\begin{prop}\label{prop:semidirect prod}
 Let $(G_0,\theta_0)$ be an oriented pro-2 group of elementary type, with $\Img(\theta_0)\not\subseteq1+4\Z_2$.
 If $G_0$ has the strong $n$-Massey vanishing property for every $n>2$, then also the semidirect product $Z\rtimes_{\theta_0}G_0$, with $Z=\langle z\rangle\simeq\Z_2$, has the strong $n$-Massey vanishing property for every $n>2$.
\end{prop}

\begin{proof}
 We keep the same notation as above.
 Let $\alpha_1,\ldots,\alpha_n$ a sequence of elements of $\rmH^1(G)$ --- by Proposition~\ref{prop:trivial strong} we may assume that they are all different to 0 --- satisfying \eqref{eq:cup 0}.
 By Lemma~\ref{lem:semidirect alpha}, one has two cases: either $\alpha_1,\ldots,\alpha_n\in\rmH^1(G_0)$; or $\alpha_i\notin\rmH^1(G_0)$ for some $i$, and hence
 $\alpha_i=\alpha_1$ if $i$ is odd, and $\alpha_i=\alpha_1+\chi_0$ if $i$ is even, and moreover $\alpha_i(z)=1$ for all $i=1,\ldots,n$.

 Suppose we are in the first case --- namely, $\alpha_i\in\rmH^1(G_0)$ for all $i$'s.
 Since $G_0$ has the strong $n$-Massey vanishing property by hypothesis, the $n$-fold Massey product $\langle\alpha_1,\ldots,\alpha_n\rangle$ vanishes (in the $\F_2$-cohomology of $G_0$).
 Hence, by Proposition~\ref{prop:masse unip} there exists a homomorphism $$\bar\rho\colon G_0\longrightarrow\dbU_{n+1}$$ such that $\bar\rho_{i,i+1}=\alpha_i$ for every $i=1,\ldots,n$. 
We construct a homomorphism $\rho\colon G\to\dbU_{n+1}$ such that $\rho\vert_{G_0}=\bar\rho$, and $Z\subseteq\Ker(\rho)$.
 Then $\rho$ satisfies all the conditions prescribed by Proposition~\ref{prop:masse unip}, so that the $n$-fold Massey product $\langle\alpha_1,\ldots,\alpha_n\rangle$ vanishes also in the $\F_2$-cohomology of $G$.
 
 Suppose now that we are in the second case --- namely, $\alpha_i\not\in\rmH^1(G_0)$ for some $i$, so that $\alpha_2=\alpha_1+\chi_0$, $\alpha_3=\alpha_1$, and so on.
 Since $\Img(\theta_0)\not\subseteq1+4\Z_2$, $G_0$ has at least a generator $u_h\in\calX$ --- and thus a generator $u_1$ --- as above.
 If $\alpha_1(u_1)=1$, we put $A=A_1$ and $B=B_1$, while if $\alpha_1(u_1)=0$, we put $A=A_2$ and $B=B_2$, where $A_i,B_i$ are as in Lemma~\ref{lemma:B2 I}.
 Recall that ${}^BA=A^{-1}$. Then 
 $$[AB,A]={}^A[B,A]\cdot[A,A]={}^A\left(A^{-2}\right)\cdot I_{n+1}=A^{-2}.$$
 Our goal is to construct a homomorphism $\rho\colon G\to\dbU_{n+1}$ satisfying all the properties prescribed in Proposition~\ref{prop:masse unip}, by assigning a suitable matrix to every generator of $G$ lying in $\calX\cup\{z\}$.
 
 We start with $z$ and we set $\rho(z)=A$.
 Then we proceed with the generators $u_1,\ldots,u_s$.
 So, pick $u_h$ with $1\leq h\leq s$, and suppose first that $\theta_0(u_h)=-1$:
 \begin{itemize}
  \item[(a)] if $\alpha_1(h_h)=\alpha_1(u_1)$, set $\rho(u_h)=B$;
  \item[(b)] if $\alpha_1(h_h)\neq\alpha_1(u_1)$, set $\rho(u_h)=AB$ --- we underline that for every $i=1,\ldots,n$, if the $(i,i+1)$-entry of $B$ is 1, then the $(i,i+1)$-entry of $AB$ is 0, and vice versa.
 \end{itemize}
 In both cases one has 
\[[\rho(u_h),\rho(z)]=[B,A]=A^{-2}=\rho(z)^{\theta_0(u_h)-1},\]
and
\[  \rho(u_h)^2=
 \begin{cases} B^2=I_{n+1} &\text{in case (a)}, \\ ABAB=A\cdot A^{-1}=I_{n+1} &\text{in case (b)};
 \end{cases}\]
moreover $\rho(u_h)_{i,i+1}=\alpha_i(u_h)$ for all $i$'s.

Suppose now that $\theta_0(u_h)=-1+2^f$ with $2\leq f<\infty$:
 \begin{itemize}
  \item[(a)] if $\alpha_1(u_h)=\alpha_1(u_1)$, set $\rho(u_h)=BC(-2^f)$;
  \item[(b)] if $\alpha_1(u_h)\neq\alpha_1(u_1)$, set $\rho(u_h)=ABC(-2^f)$.
 \end{itemize}
 Since $C(-2^f)\in(\dbU_{n+1})_{(3)}$, the $(i,i+1)$-entry of $BC(-2^f)$, respectively of $ABC(-2^f)$, is the same as in $B$, respectively in $AB$, namely, $\alpha_i(u_h)$.
 In both cases one has 
\[[\rho(u_h),\rho(z)]=[BC(-2^f),A]={}^B\left(A^{-2^f}\right)\cdot A^{-2}=A^{2^f-2}=\rho(z)^{\theta_0(u_h)-1},\]
and moreover $\rho(u_h)_{i,i+1}=\alpha_i(u_h)$ for all $i$'s.

We are done with the generators $u_1,\ldots,u_s$, so we proceed with the generators $v_1,\ldots,v_t$.
Recall that the value of $\alpha_i(v_l)$ is the same for all $i$'s.
Pick $v_l$, and suppose that $\theta_0(v_l)=1+2^f$ with $2\leq f<\infty$:
 \begin{itemize}
  \item[(a)] if $\alpha_1(v_l)=0$, set $\rho(v_l)=C(2^f)$, whose $(i,i+1)$-entry is 0 for all $i$'s;
  \item[(b)] if $\alpha_1(v_l)=1$, set $\rho(v_l)=AC(2^f)$, whose $(i,i+1)$-entry is 1 for all $i$'s.
  \end{itemize}
 In both cases one has 
 $$[\rho(v_l),\rho(z)] =[C(2^f),A]=A^{2^f}=\rho(z)^{\theta_0(v_l)-1},$$
and moreover $\rho(v_l)_{i,i+1}=\alpha_i(v_l)$ for all $i$'s.

Suppose now that $\theta_0(v_l)=1$:
  \begin{itemize}
  \item[(a)] if $\alpha_1(v_l)=0$, set $\rho(v_l)=I_{n+1}$;
  \item[(b)] if $\alpha_1(v_l)=1$, set $\rho(v_l)=A$.
    \end{itemize}
 In both cases one has $[\rho(v_l),\rho(z)] =I_{n+1}=\rho(z)^{\theta_0(v_l)-1}$, and  
\[ 
[\rho(y),\rho(v_l)]= \begin{cases}[\rho(y),I_{n+1}]=I_{n+1}=\rho(v_l)^{\theta_0(y)-1} & \text{in case (a)},\\
 [\rho(y),\rho(z)]=\rho(z)^{\theta_0(y)-1}=\rho(v_l)^{\theta_0(y)-1} & \text{in case (b)},
\end{cases}
\]
for any $y\in\calX$,
and moreover $\rho(v_l)_{i,i+1}=\alpha_i(v_l)$ for all $i$'s.

Altogether, the matrices $\rho(y)$, with $y\in\calX$, and $\rho(z)$, satisfy the defining relations \eqref{eq:rel semidirect u}--\eqref{eq:rel semidirect v2} of $G_0$ (see also Remark~\ref{rem:factors relations semidirect}), and the relations $[y,z]z^{1-\theta_0(y)}=1$.
Therefore, the assignment 
$$z\longmapsto\rho(z),\qquad y\longmapsto \rho(y),\text{ with }y\in\calX$$
induces a homomorphism $\rho\colon G\to \dbU_{n+1}$, satisfying $\rho_{i,i+1}=\alpha_i$ for all $i=1,\ldots,n$.
Hence, the $n$-fold Massey product $\langle\alpha_1,\ldots,\alpha_n\rangle$ vanishes by Proposition~\ref{prop:masse unip}.
 \end{proof}
 

\subsection{Proofs of Theorems~1.2--1.3}

\begin{proof}[Proof of Theorem~1.3]
 We follow the recursive construction of an oriented pro-2 group of elementary type.
 
 If $(G,\theta)$ is an oriented pro-2 group with $G$ a finitely generated free pro-2 group, then $G$ has the strong $n$-Massey vanishing property for all $n>2$ by Proposition~\ref{prop:free}.
 If $(G,\theta)$ is an oriented pro-2 group with $G$ a Demushkin pro-2 group, then $G$ has the strong $n$-Massey vanishing property for all $n>2$ by Proposition~\ref{prop:Demushkin}.
 So, all building bricks of an oriented pro-2 group of elementary type have the strong $n$-Massey vanishing property for all $n>2$.
 
 If $(G_1,\theta_1),(G_2,\theta_2)$ are oriented pro-2 groups with $G_1,G_2$ both satisfying the strong $n$-Massey vanishing property for all $n>2$, then also the free pro-2 product $G_1\ast G_2$ has the strong $n$-Massey vanishing property for all $n>2$ by Proposition~\ref{prop:free}.
 
 If $(G_0,\theta_0)$ is an oriented pro-2 groups with $G_0$ satisfying the strong $n$-Massey vanishing property for all $n>2$, then also the semidirect product $Z\rtimes_{\theta_0} G_0$, with $Z\simeq \Z_2$, has the strong $n$-Massey vanishing property for all $n>2$ by Proposition~\ref{prop:semidirect prod}.
 
 Finally, if $H$ is an open subgroup of $G$ then it is finitely generated (cf., e.g., \cite[Prop.~1.7]{ddsms}), and thus 
 also the oriented pro-2 group $(H,\theta\vert_H)$ is of elementary type, and it has the strong $n$-Massey vanishing property for all $n>2$ by the argument above.
\end{proof}

\begin{proof}[Proof of Theorem~1.2]
 By \cite{minac:pyt}, if $\K$ is a Pythagorean field with a finite number of square classes, then $(G_{\K}(2),\theta_{\K})$ is an oriented pro-2 group of elementary type, and thus the maximal pro-2 Galois group $G_{\K}(2)$ has the strong $n$-Massey vanishing property for all $n>2$ by Theorem~\ref{thm:ET intro}.
 
 If $\mathbb{L}/\K$ is a finite extension, then the maximal pro-2 Galois group $G_{\mathbb{L}}(2)$ is an open subgroup of $G_{\K}(2)$, of index the maximal 2-power dividing $[\mathbb{L}:\K]$.
 Therefore, $G_{\mathbb{L}}(2)$ has the strong $n$-Massey vanishing property for all $n>2$ by Theorem~\ref{thm:ET intro}.
\end{proof}

Corollary~\ref{cor:etc} follows from Theorem~\ref{thm:ET intro} and Proposition~\ref{prop:etc fields}, with the same argument for the finite extensions as in the proof of Theorem~\ref{thm:main}.

By \cite[Example~A.15]{GPM}, and by \cite[Thm.~6.3]{MerSca1}, one knows that there are fields, with an infinite number of square classes, whose maximal pro-$2$ Galois group does not have the strong $n$-Massey vanishing property for $n=4$ --- for example, $\Q$ is one of them.
We ask whether, instead, a field with a finite number of square classes has maximal pro-2 Galois group with the strong $n$-Massey vanishing property for every $n>2$ --- as done in \cite[Question~1.5]{cq:massey} assuming further that the field contains $\sqrt{-1}$.

\begin{question}\label{ques:p2}
Let $\K$ be a field with a finite number of square classes.
Does the maximal pro-$2$ Galois group $G_{\K}(2)$ have the strong $n$-Massey vanishing property for every $n>2$?
\end{question}



\begin{bibdiv}
\begin{biblist}



\bib{BCQ:AbsGalType}{article}{
   author={Blumer, S.},
   author={Cassella, A.},
   author={Quadrelli, C.},
   title={Groups of $p$-absolute Galois type that are not absolute Galois
   groups},
   journal={J. Pure Appl. Algebra},
   volume={227},
   date={2023},
   number={4},
   pages={Paper No. 107262},
}

   



\bib{Demushkin}{article}{
   author={Demu\v{s}kin, S.P.},
   title={The group of a maximal $p$-extension of a local field},
   language={Russian},
   journal={Izv. Akad. Nauk SSSR Ser. Mat.},
   volume={25},
   date={1961},
   pages={329--346},}
	
\bib{ddsms}{book}{
   author={Dixon, J.D.},
   author={du Sautoy, M.P.F.},
   author={Mann, A.},
   author={Segal, D.},
   title={Analytic pro-$p$ groups},
   series={Cambridge Studies in Advanced Mathematics},
   volume={61},
   edition={2},
   publisher={Cambridge University Press, Cambridge},
   date={1999},
   pages={xviii+368},
   isbn={0-521-65011-9},
}

\bib{dwyer}{article}{
   author={Dwyer, W.G.},
   title={Homology, Massey products and maps between groups},
   journal={J. Pure Appl. Algebra},
   volume={6},
   date={1975},
   number={2},
   pages={177--190},
   issn={0022-4049},
}

\bib{ido:etc}{article}{
   author={Efrat, I.},
   title={Orderings, valuations, and free products of Galois groups},
   journal={Sem. Structure Alg\'ebriques Ordonn\'ees, Univ. Paris VII},
   date={1995},
}

\bib{ido:etc2}{article}{
   author={Efrat, I.},
   title={Pro-$p$ Galois groups of algebraic extensions of $\mathbf{Q}$},
   journal={J. Number Theory},
   volume={64},
   date={1997},
   number={1},
   pages={84--99},
}

\bib{efrat:small}{article}{
   author={Efrat, I.},
   title={Small maximal pro-$p$ Galois groups},
   journal={Manuscripta Math.},
   volume={95},
   date={1998},
   number={2},
   pages={237--249},
   issn={0025-2611},
}


\bib{ido:Hasse}{article}{
   author={Efrat, I.},
   title={A Hasse principle for function fields over PAC fields},
   journal={Israel J. Math.},
   volume={122},
   date={2001},
   pages={43--60},
   issn={0021-2172},
}

\bib{ido:Massey}{article}{
   author={Efrat, I.},
   title={The Zassenhaus filtration, Massey products, and homomorphisms of
   profinite groups},
   journal={Adv. Math.},
   volume={263},
   date={2014},
   pages={389--411},
}

\bib{idomatrix}{article}{
   author={Efrat, I.},
   title={The lower $p$-central series of a free profinite group and the
   shuffle algebra},
   journal={J. Pure Appl. Algebra},
   volume={224},
   date={2020},
   number={6},
   pages={106260, 13},
   issn={0022-4049},
}

\bib{EM:Massey}{article}{
   author={Efrat, I.},
   author={Matzri, E.},
   title={Triple Massey products and absolute Galois groups},
   journal={J. Eur. Math. Soc. (JEMS)},
   volume={19},
   date={2017},
   number={12},
   pages={3629--3640},
   issn={1435-9855},
}

\bib{idojan:descending}{article}{
   author={Efrat, I.},
   author={Mina\v{c}, J.},
   title={On the descending central sequence of absolute Galois groups},
   journal={Amer. J. Math.},
   volume={133},
   date={2011},
   number={6},
   pages={1503--1532},
}


\bib{eq:kummer}{article}{
   author={Efrat, I.},
   author={Quadrelli, C.},
   title={The Kummerian property and maximal pro-$p$ Galois groups},
   journal={J. Algebra},
   volume={525},
   date={2019},
   pages={284--310},
   issn={0021-8693},
}

\bib{elmanlam}{article}{
   author={Elman, R.},
   author={Lam, T.Y.},
   title={Quadratic forms over formally real fields and pythagorean fields},
   journal={Amer. J. Math.},
   volume={94},
   date={1972},
   pages={1155--1194},
   issn={0002-9327},}

\bib{friedjarden}{book}{
   author={Fried, M. D.},
   author={Jarden, M.},
   title={Field arithmetic},
   series={Ergebnisse der Mathematik und ihrer Grenzgebiete. 3. Folge. A
   Series of Modern Surveys in Mathematics},
   volume={11},
   edition={4},
   note={Revised by M.~Jarden},
   publisher={Springer, Cham},
   date={2023},
   pages={xxxi+827},
   isbn={978-3-031-28019-1},
   isbn={978-3-031-28020-7},
}

\bib{jochen}{article}{
   author={G\"{a}rtner, J.},
   title={Higher Massey products in the cohomology of mild pro-$p$-groups},
   journal={J. Algebra},
   volume={422},
   date={2015},
   pages={788--820},
   issn={0021-8693},
}


\bib{GPM}{article}{
   author={Guillot, P.},
   author={Mina\v{c}, J.},
   author={Topaz, A.},
   title={Four-fold Massey products in Galois cohomology},
   note={With an appendix by O.~Wittenberg},
   journal={Compos. Math.},
   volume={154},
   date={2018},
   number={9},
   pages={1921--1959},
}


	\bib{PJ}{article}{
   author={Guillot, P.},
   author={Mina\v{c}, J.},
   title={Extensions of unipotent groups, Massey products and Galois theory},
   journal={Adv. Math.},
   volume={354},
   date={2019},
   pages={article no. 106748},
}
	

\bib{HW:Massey}{article}{
   author={Harpaz, Y.},
   author={Wittenberg, O.},
   title={The Massey vanishing conjecture for number fields},
   journal={Duke Math. J.},
   volume={172},
   date={2023},
   number={1},
   pages={1--41},
}

\bib{hopwick}{article}{
   author={Hopkins, M.J.},
   author={Wickelgren, K.G.},
   title={Splitting varieties for triple Massey products},
   journal={J. Pure Appl. Algebra},
   volume={219},
   date={2015},
   number={5},
   pages={1304--1319},
}

\bib{jacob1}{article}{
   author={Jacob, B.},
   title={On the structure of Pythagorean fields},
   journal={J. Algebra},
   volume={68},
   date={1981},
   number={2},
   pages={247--267},
   issn={0021-8693},}
   
\bib{jacob2}{article}{
   author={Jacob, B.},
   title={The Galois cohomology of Pythagorean fields},
   journal={Invent. Math.},
   volume={65},
   date={1981/82},
   number={1},
   pages={97--113},
   issn={0020-9910},}



\bib{jacobware}{article}{
   author={Jacob, B.},
   author={Ware, R.},
   title={A recursive description of the maximal pro-$2$ Galois group via
   Witt rings},
   journal={Math. Z.},
   volume={200},
   date={1989},
   number={3},
   pages={379--396},
   issn={0025-5874},}

	
\bib{labute}{article}{
   author={Labute, J.P.},
   title={Classification of Demushkin groups},
   journal={Canadian J. Math.},
   volume={19},
   date={1967},
   pages={106--132},
   issn={0008-414X},
}
		
	\bib{lam}{book}{
   author={Lam, T.Y.},
   title={Introduction to quadratic forms over fields},
   series={Grad. Stud. Math.},
   volume={67},
   publisher={American Mathematical Society, Providence, RI},
   date={2005},
}

		
\bib{LLSWW}{article}{
   author={Lam, Y.H.J.},
   author={Liu, Y.},
   author={Sharifi, R.T.},
   author={Wake, P.},
   author={Wang, J.},
   title={Generalized Bockstein maps and Massey products},
   date={2023},
   journal={Forum Math. Sigma},
   volume={11},
   date={2023},
   pages={Paper No. e5},
}


\bib{marshall}{article}{
   author={Marshall, M.},
   title={The elementary type conjecture in quadratic form theory},
   conference={
      title={Algebraic and arithmetic theory of quadratic forms},
   },
   book={
      series={Contemp. Math.},
      volume={344},
      publisher={Amer. Math. Soc., Providence, RI},
   },
   date={2004},
   pages={275--293},}

\bib{eli:Massey}{unpublished}{
   author={Matzri, E.},
   title={Triple Massey products in Galois cohomology},
   date={2014},
   note={Preprint, available at {\tt arXiv:1411.4146}},
}

\bib{eli2}{article}{
   author={Matzri, E.},
   title={Triple Massey products of weight $(1,n,1)$ in Galois cohomology},
   journal={J. Algebra},
   volume={499},
   date={2018},
   pages={272--280},
}

\bib{eli3}{article}{
   author={Matzri, E.},
   title={Higher triple Massey products and symbols},
   journal={J. Algebra},
   volume={527},
   date={2019},
   pages={136--146},
   issn={0021-8693},}
		
	

\bib{MerSca1}{unpublished}{
   author={Merkurjev, A.},
   author={Scavia, F.},
   title={Degenerate fourfold Massey products over arbitrary fields},
   date={2022},
   note={Preprint, available at {\tt arXiv:2208.13011}},
}

\bib{MerSca2}{unpublished}{
   author={Merkurjev, A.},
   author={Scavia, F.},
   title={The Massey Vanishing Conjecture for fourfold Massey products modulo 2},
   date={2023},
   note={Preprint, available at {\tt arXiv:2301.09290}},
}

\bib{minac:pyt}{article}{
   author={Mina\v{c}, J.},
   title={Galois groups of some $2$-extensions of ordered fields},
   journal={C. R. Math. Rep. Acad. Sci. Canada},
   volume={8},
   date={1986},
   number={2},
   pages={103--108},
}
		
\bib{MPQT}{article}{
   author={Mina\v{c}, J.},
   author={Pasini, F.W.},
   author={Quadrelli, C.},
   author={T\^{a}n, N.D.},
   title={Koszul algebras and quadratic duals in Galois cohomology},
   journal={Adv. Math.},
   volume={380},
   date={2021},
   pages={article no. 107569},
}
	
		
   



   \bib{mt:conj}{article}{
   author={Mina\v{c}, J.},
   author={T\^{a}n, N.D.},
   title={The kernel unipotent conjecture and the vanishing of Massey
   products for odd rigid fields},
   journal={Adv. Math.},
   volume={273},
   date={2015},
   pages={242--270},
}

\bib{mt:docu}{article}{
   author={Mina\v{c}, J.},
   author={T\^{a}n, N.D.},
   title={Triple Massey products over global fields},
   journal={Doc. Math.},
   volume={20},
   date={2015},
   pages={1467--1480},
}
		
     
\bib{MT:Masseyall}{article}{
   author={Mina\v{c}, J.},
   author={T\^{a}n, N.D.},
   title={Triple Massey products vanish over all fields},
   journal={J. London Math. Soc.},
   volume={94},
   date={2016},
   pages={909--932},
}


   
   \bib{JT:U4}{article}{
   author={Mina\v{c}, J.},
   author={T\^{a}n, N.D.},
   title={Counting Galois $\dbU_4(\F_p)$-extensions using Massey products},
   journal={J. Number Theory},
   volume={176},
   date={2017},
   pages={76--112},
   issn={0022-314X},
}

\bib{mt:Massey}{article}{
   author={Mina\v{c}, J.},
   author={T\^{a}n, N.D.},
   title={Triple Massey products and Galois theory},
   journal={J. Eur. Math. Soc. (JEMS)},
   volume={19},
   date={2017},
   number={1},
   pages={255--284},
   issn={1435-9855},
}

\bib{nsw:cohn}{book}{
   author={Neukirch, J.},
   author={Schmidt, A.},
   author={Wingberg, K.},
   title={Cohomology of number fields},
   series={Grundlehren der Mathematischen Wissenschaften},
   volume={323},
   edition={2},
   publisher={Springer-Verlag, Berlin},
   date={2008},
   pages={xvi+825},
   isbn={978-3-540-37888-4},}
  

\bib{palquick:Massey}{unpublished}{
   author={P\'al, A.},
   author={Quick, G.},
   title={Real projective groups are formal},
   date={2022},
   note={Preprint, available at {\tt arXiv:2206.14645}},
}
  
\bib{pal:Massey}{unpublished}{
   author={P\'al, A.},
   author={Szab\'o, E.},
   title={The strong Massey vanishing conjecture for fields with virtual cohomological dimension at most 1},
   date={2020},
   note={Preprint, available at {\tt arXiv:1811.06192}},
}
  
  

   

\bib{cq:massey}{article}{
   author={Quadrelli, C.},
   title={Massey products in Galois cohomology and the elementary type conjecture},
   date={2024},
   journal={J. Number Theory},
   volume={258},
   pages={40--65}
}
  
\bib{qw:cyc}{article}{
   author={Quadrelli, C.},
   author={Weigel, Th.S.},
   title={Profinite groups with a cyclotomic $p$-orientation},
   date={2020},
   volume={25},
   journal={Doc. Math.},
   pages={1881--1916}
   }

	
  
\bib{ribzal}{book}{
   author={Ribes, L},
   author={Zalesski\u{\i}, P.A.},
   title={Profinite groups},
   series={Ergebnisse der Mathematik und ihrer Grenzgebiete. 3. Folge. A
   Series of Modern Surveys in Mathematics},
   volume={40},
   edition={2},
   publisher={Springer-Verlag, Berlin},
   date={2010},
   pages={xvi+464},
}
	

	
   
   \bib{serre:Demushkin}{article}{
   author={Serre, J-P.},
   title={Structure de certains pro-$p$-groupes (d'apr\`es Demu\v{s}kin)},
   language={French},
   conference={
      title={S\'{e}minaire Bourbaki, Vol. 8},
   },
   book={
      publisher={Soc. Math. France, Paris},
   },
   date={1995},
   pages={Exp. No. 252, 145--155},
}
		
\bib{serre:galc}{book}{
   author={Serre, J-P.},
   title={Galois cohomology},
   series={Springer Monographs in Mathematics},
   edition={Corrected reprint of the 1997 English edition},
   note={Translated from the French by Patrick Ion and revised by the
   author},
   publisher={Springer-Verlag, Berlin},
   date={2002},
   pages={x+210},
   isbn={3-540-42192-0},}
	
	

	


   
   
   \bib{vogel}{report}{
   author={Vogel, D.},
   title={Massey products in the Galois cohomology of number fields},
   date={2004},
   note={PhD thesis, University of Heidelberg},
   eprint={http://www.ub.uni-heidelberg.de/archiv/4418},
}
	
%



\bib{wadsworth}{article}{
   author={Wadsworth, A.},
   title={$p$-Henselian field: $K$-theory, Galois cohomology, and graded
   Witt rings},
   journal={Pacific J. Math.},
   volume={105},
   date={1983},
   number={2},
   pages={473--496},
}
		
\bib{ware}{article}{
   author={Ware, R.},
   title={Galois groups of maximal $p$-extensions},
   journal={Trans. Amer. Math. Soc.},
   volume={333},
   date={1992},
   number={2},
   pages={721--728},
}


\bib{wick}{article}{
   author={Wickelgren, K.},
   title={Massey products $\langle y$, $x$, $x$, $\ldots$, $x$, $x$,
   $y\rangle$ in Galois cohomology via rational points},
   journal={J. Pure Appl. Algebra},
   volume={221},
   date={2017},
   number={7},
   pages={1845--1866},
   issn={0022-4049},
}
	

\end{biblist}
\end{bibdiv}
\end{document}